\documentclass[12pt]{amsart}

\usepackage{amssymb}
\usepackage{tikz-cd}

\usepackage{enumitem}

\makeatletter
\@namedef{subjclassname@2020}{%
  \textup{2020} Mathematics Subject Classification}
\makeatother

\usepackage[T1]{fontenc}

\newtheorem{theorem}{Theorem}[section]

\newtheorem{proposition}[theorem]{Proposition}

\newtheorem{algorithm}[theorem]{Algorithm}

\theoremstyle{definition}

\newtheorem{remark}[theorem]{Remark}
\newtheorem{example}[theorem]{Example}

\numberwithin{equation}{section}

\newcommand{\closure}[2][3]{%
  {}\mkern#1mu\overline{\mkern-#1mu#2}}

\begin{document}


\baselineskip=17pt


\title[An Extension to the Gusi\'c-Tadi\'c Specialization Criterion]{An Extension to the Gusi\'c-Tadi\'c Specialization Criterion}

\author[T. R. Billingsley]{Tyler Raven Billingsley}
\address{ Rose-Hulman Institute of Technology\\Department of Mathematics \\
5500 Wabash Ave\\
Terre Haute, IN 47803}
\email{billings@rose-hulman.edu}

\date{09/28/2022}

\begin{abstract}
Let $E/\mathbb Q(t)$ be an elliptic curve and let $t_0 \in \mathbb Q$ be a rational number for which the specialization $E_{t_0}$ is an elliptic curve. In 2015, Gusi\'c and Tadi\'c gave an easy-to-check criterion, based only on a Weierstrass equation for $E/\mathbb Q(t)$, that is sufficient to conclude that the specialization map at $t_0$ is injective. The criterion critically requires that $E$ has nontrivial $\mathbb Q(t)$-rational 2-torsion points. In this article, we explain how the criterion can be used in some cases where this requirement is not satisfied and provide some examples.
\end{abstract}

\subjclass[2020]{Primary 11; Secondary 14}


\maketitle

\section{Introduction}

Let $C$ be a (complete nonsingular) curve defined over a number field $k$ with function field $k(C)$. Let $E/k(C)$ be an elliptic curve defined by the Weierstrass equation $$y^2 = x^3 +A(t)x+B(t), \qquad A(t), B(t) \in k(C).$$ For any $t_0 \in C(k)$ such that the discriminant $4A(t)^3+27B(t)^2$ of $E$ does not vanish or have a pole at $t_0$, we define the elliptic curve $E_{0}/k$ using the Weierstrass equation $y^2 = x^3 + A(t_0)x+B(t_0)$. The specialization map at $t_0$ is the map 
$$\sigma_{t_0} : E(k(C)) \to E_{t_0}(k).$$
which takes the point $(x(t), y(t)) \in E(k(C))$ to $(x(t_0),y(t_0)) \in E_{t_0}(k)$. It is in fact a group homomorphism; that is, it respects the standard chord-and-tangent group laws on the domain and codomain. It is natural to ask what information can be extracted about the relationship between $E(k(C))$ and $E_{t_0}(k)$ through this homomorphism.

In 1952, N\'eron proved a theorem regarding specialization, of which the following result is a special case.
\begin{theorem} \cite{Neron-Spec} Let $k$ be a number field and let $E/k(t)$ be an elliptic curve. Then for infinitely many $t_0 \in k$, the specialization map
$$\sigma_{t_0} : E(k(t)) \to E_{t_0}(k)$$
is injective.
\end{theorem}
Thirty years later, Silverman \cite{Silverman-Spec} improved on N\'eron's result by proving that all but finitely many maps $\sigma_{t_0}$ are injective by showing that the set of $t_0 \in k$ for which injectivity fails is a set of bounded height. While Silverman's result is effective, it might not give a practical way to find this finite set, and so the author knows of no nontrivial examples (that is, no examples of non-split elliptic curves) for which these finitely many non-injective specialization maps have been identified. Methods which allow us to work concretely and effectively with specialization maps allow us to inch closer to computing such an example.

In 2021 \cite{Billingsley-algorithm}, the author discussed an algorithm, based on the proof of Theorem 1, which allows one to effectively check when (most) specialization maps are injective on a given subgroup of the Mordell-Weil group. While this algorithm works nicely, one downfall is that one requires explicit generators of a subgroup in order to apply it. One may use the criteria of Gusi\'c and Tadi\'c \cite{GusicTadic} to check if specialization is injective without knowing generators of the Mordell-Weil group. The method, based on 2-descent, only works as long as your elliptic curve has a $k(t)$-rational 2-torsion point, so it leaves the question open for how to work with the situation with unknown generators and no $k(t)$-rational 2-torsion.

The aim of the present article is to generalize Gusi\'c and Tadi\'c's criteria to apply in some cases where there is no $k(t)$-rational 2-torsion. We apply this to an example that the previous criterion could not address, then we combine the methods of \cite{Billingsley-algorithm} and Gusi\'c-Tadi\'c to address another example which neither method could work with originally.

The calculations in this article were greatly aided by Sage \cite{sagemath}.

\section{Motivation and Preliminaries} \label{sec:weakmw}

We begin by recalling some effective specialization results that will be needed in \S \ref{sec:2divcurves}. For details, see \cite{Billingsley-algorithm}. First, we will need the following group-theoretic criterion to show that specialization maps are injective.

\begin{proposition} \label{prop:grouptheory} \cite{Serre-MordellWeil}
Let $n$ be a positive integer and let $\phi: M \to N$ be a homomorphism of abelian groups with the following properties.
	\begin{enumerate}
	\item $M$ is finitely generated.
	\item The induced map $\bar\phi: M/nM \to N/nN$ is injective.
	\item $\phi|_{M[n]}$ gives an isomorphism $M[n] \cong N[n]$.
	\item $\phi|_{M_\text{tors}}$ is injective.
	\end{enumerate}
Then $\phi$ is injective.
\end{proposition}

For our applications, fix an elliptic curve $E$ over $\mathbb Q(t)$. Set $\phi = \sigma_{t_0}$ to be the specialization homomorphism for a fixed $t_0 \in \mathbb Q$, $M = E(\mathbb Q(t))$, $N = E_{t_0}(\mathbb Q)$ and $n \geq 2$ a positive integer. In this case, conditions 1 and 4 above are always true; indeed, condition 1 is the function field version of the Mordell-Weil Theorem \cite[Chapter III \S 6]{Silverman2}, and condition 4 follows from basic results on formal groups of elliptic curves and their relationship to reduction mod $p$ found in Silverman \cite[Chapter VII \S 3]{Silverman1}. So checking conditions 2 and 3 suffice to show that a given specialization map is injective. We provide two examples below that show that neither of the two conditions imply the other in general, then we provide a third example that shows that condition 2 is not necessary for specialization to be injective.

\begin{example} \label{ex:torsion_needed} (2 holds, but not 3)
Let $E: y^2 = x^3-(t^2+27)x+(10t^2+48t+90)$, $\phi = \sigma_{30}$, $M = E(\mathbb Q(t))$, $N = E_{30}(\mathbb Q)$ and $n = 2$. In \cite{Shioda-Examples}, this elliptic curve is shown to have Mordell-Weil rank 4 and to have no nontrivial torsion points over $\mathbb{Q}(t)$. The Mordell-Weil group is generated by the four points $P_1 = (9, t+24), P_2 = (6,2t+12), P_3 = (1,3t+8)$ and $P_4 = (t+3,4t+6)$, so a complete set of representatives for the nonidentity cosets of $2E(\mathbb Q(t))$ in $E(\mathbb Q(t))$ is 
$$\left\{\sum_{i \in C} P_i \mid C \subset \{1,2,3,4\}, C \neq \emptyset\right\}.$$
One can check (for instance, using the {\upshape \texttt{EllipticCurve}} method {\upshape \texttt{division\_points(2)}} in Sage) that for $t=30$ the specialization of each of these 15 points is not divisible by 2 in $E_{30}$. Thus 
$$\closure[1]{\sigma_{30}}: E(\mathbb Q(t))/2E(\mathbb Q(t)) \to E_{30}(\mathbb Q)/2E_{30}(\mathbb Q)$$
 is injective. However, the Mordell-Weil group of $E_{30}$ is $\mathbb{Z}^3 \times \mathbb{Z}/2\mathbb{Z}$, so $\sigma_{30}$ cannot be injective. In particular, condition 3 of Proposition \ref{prop:grouptheory} does not hold, but condition 2 does hold.
\end{example}

\begin{example} \label{ex:3not2} (3 holds, but not 2)
Let $E: y^2 = x^3-t^2x+t^2, \phi = \sigma_{2}$, $M = E(\mathbb Q(t))$, $N = E_{2}(\mathbb Q)$ and $n = 2$. One can check that $E(\mathbb Q(t)) \cong \mathbb Z^2$ with generators
$$P = (t,t), Q = (0,t).$$
Then, using Sage, one can check that the specialization $E_2$ has $E_2(\mathbb Q) \cong \mathbb Z$. Hence $\closure[1]{\sigma_2}$ is a map from a group of order 4 to a group of order 2, so $\closure[1]{\sigma_2}$ cannot be injective.
\end{example}

\begin{example} \label{ex:linear_algebra}
Let $E$ be the elliptic curve from Example \ref{ex:3not2} and set $t_0 = 27$. On the elliptic curve $E_{27}: y^2 = x^3 - 729x + 729$, notice that
$$[2](-9,81) = (27,27) = P_{27},$$
so condition 2 of Proposition \ref{prop:grouptheory} fails. A check using Sage shows that $E_{27}(\mathbb Q) \cong \mathbb Z^2$ with generators $R_1 = (-9,81)$ and $R_2 = (-27,27)$. Now $P_{27} = 2R_1$ and $Q_{27} = -(2R_1+R_2)$, meaning the matrix of the specialization map $\sigma_{27}$ with respect to the ordered bases $\{P,Q\}$ and $\{R_1,R_2\}$ is
$$\begin{bmatrix}
2 & -2 \\
0 & -1
\end{bmatrix}.$$
The determinant of this matrix is $-2 \neq 0$, so $\sigma_{27}$ is injective.
\end{example}

Proposition \ref{prop:grouptheory} can be applied to yield an algorithm for checking if a specialization map is injective.

\begin{algorithm} \label{algorithm} \cite{Billingsley-algorithm}
Let $E/\mathbb Q(t)$ be an elliptic curve given by a Weierstrass equation and let $M < E(\mathbb Q(t))$ be a subgroup such that the inclusion $M \subset E(\mathbb Q(t))$ satisfies the hypotheses of Proposition \ref{prop:grouptheory} for some $n \geq 2$. Then for all $t_0 \in \mathbb Q$ except those contained in a certain subset of $\mathbb Q$, there is an algorithm which shows that the specialization map $\sigma_{t_0}|_M$ is injective.
\end{algorithm}

The details of this algorithm are not important for our purposes, but it is worth emphasizing that the algorithm requires knowing generators of $M/nM$ for some $n$. For the case $M = E(\mathbb Q(t))$, finding such generators in general is an open question.

Finally, we recall the results of Gusi\'c and Tadi\'c \cite{GusicTadic} on which this article builds. The following first statement is modified from the original source to reflect Stoll's recasting of the criteria to include their close connection to 2-descent \cite{Stoll-Spec}.

\begin{theorem} \cite{GusicTadic} \label{thm:GT_2-descent} Let $k$ be a number field and let $R$ be a UFD containing the ring of integers $\mathcal O_k$ with a finitely generated unit group. Let $E/k(t)$ be an elliptic curve given by a Weierstrass equation
$$E: y^2 = (x-e_1)(x-e_2)(x-e_3)$$
with $e_i \in R[t]$ for each $i$. Let $G$ be the subgroup of $k(t)^*/(k(t)^*)^2 \times k(t)^*/(k(t)^*)^2$ generated by tuples with entries coming from the irreducible factors of
\begin{equation} \tag{A} \label{eqn:disc_factors} e_1-e_2, \qquad e_2-e_3, \qquad e_1-e_3 \end{equation}
in (the UFD) $R[t]$ and the units of $R$. If each nonidentity $(f(k(t)^*)^2,g(k(t)^*)^2) \in G$ has the property that $f(t_0)$ or $g(t_0)$ is not a square in $k$, then the specialization map at $t_0$ is injective.
\end{theorem}

The second criterion below is modified to reflect its relationship to descent by 2-isogeny.

\begin{theorem} \cite{GusicTadic} \label{thm:GT_2-isogeny} Let $k$ and $R$ be as in Theorem \ref{thm:GT_2-descent}. Let $E/k(t)$ be an elliptic curve given by the Weierstrass equation
$$E: y^2 = x^3+ax^2+bx$$
with $a,b \in R[t]$, and additionally suppose that $a^2-4b$ is not a square in $k$. Let $G_\phi, G_{\hat\phi}$ be the subgroups of $k(t)^*/(k(t)^*)^2$ generated by the irreducible factors of
$$ b \text{ and } a^2-4b$$
in $R[t]$, respectively, and the units of $R$. If each nonidentity $f(k(t)^*)^2 \in G_\phi$ and $g(k(t)^*)^2 \in G_{\hat\phi}$ has the property that $f(t_0)$ and $g(t_0)$ are not squares in $k$, then the specialization map at $t_0$ is injective.
\end{theorem}

\begin{remark} \label{rem:discriminant} In Theorem \ref{thm:GT_2-descent}, note that the discriminant of $E$ is 
$$16((e_3-e_1)(e_3-e_2)(e_2-e_1))^2.$$
Hence the factors considered there are a subset of those of the discriminant of $E$.

Additionally, in Theorem \ref{thm:GT_2-isogeny}, $E$ has discriminant $-16b^2(a^2-4b)$ and the dual curve $E'$ has discriminant $2^{10}b(a^2-4b)^2$. Hence the factors considered there are a subset of those of either discriminant. In particular, recall that, for an elliptic curve with a 2-torsion point that is not $(0,0)$, moving the 2-torsion point to $(0,0)$ does not change the discriminant. Thus we can apply Theorem \ref{thm:GT_2-isogeny} to any elliptic curve with exactly one nontrivial 2-torsion point by considering factors of the discriminant. Further, since we have demonstrated that what matters is the factors of $b$ and $a^2-4b$, we may ignore the $16$ in front of the discriminant of $E$; that is, we can take the factors of $\Delta_E$/16 and still obtain suitable generators for $G_\phi$ and $G_{\hat\phi}$.

\end{remark}

\section{Rational 2-division Curves} \label{sec:2divcurves}

Notice that all methods discussed previously show that a specialization map is injective by ultimately showing that the induced map on $E(K)/2E(K)$ is injective. As Example \ref{ex:torsion_needed} shows, this alone does not imply that specialization is injective. Additionally, Example \ref{ex:linear_algebra} shows that specialization can be injective without this condition holding true. Provided that certain polynomials define a rational curve over $\mathbb Q$, we provide a way to get around this for certain specializations. In particular, if the 2-torsion curve is rational, this approach allows us to apply the criterion of Gusi\'c to certain specializations of elliptic curves without rational 2-torsion points. The following Proposition provides the setup to apply Gusi\'c's criterion in this new way.

\begin{proposition}\label{prop:torsion_division} Let $k$ be a number field. Let $E/k(t)$ be an elliptic curve given by the Weierstrass equation
$$y^2 = x^3+A(t)x+B(t)$$
such that the 2-torsion curve
$$C_2: a^3+A(t)a+B(t)=0$$
is irreducible and rational over $k$. Fix an isomorphism of function fields
\begin{align*}
k(C_2) &\cong k(\alpha) \\
t &\mapsto u(\alpha) \\
a &\mapsto v(\alpha).
\end{align*}
Then the elliptic curve 
$$E': y^2 = x^3+A(u(\alpha))x+B(u(\alpha))$$
defined over $k(\alpha)$ has the following properties.
\begin{enumerate}
	\item $E'(k(\alpha))$ has the nontrivial 2-torsion point $(v(\alpha),0)$.
	\item The function field isomorphism gives an embedding $E(k(t)) \subset E'(k(\alpha)).$
	\item Let $\alpha_0 \in k$ and set $t_0 = u(\alpha_0)$. Let $M < E(k(t))$ be a subgroup. If the specialization map $\sigma'_{\alpha_0}$ for $E'$ is injective on the image of $M$ via the embedding above, then the specialization map $\sigma_{t_0}|_M$ for $E$ is injective.
\end{enumerate}
\end{proposition}
\begin{proof}
\begin{enumerate}
\item Note that $x^3+A(t)x+B(t)$ vanishes at $x = a$. By applying the isomorphism of function fields above, we see that $x^3+A(u(\alpha))x+B(u(\alpha))$ vanishes at $x = v(\alpha)$. However, $v(\alpha) \in k(\alpha)$, so $E'$ has the $k(\alpha)$-rational 2-torsion point $(v(\alpha),0)$.
\item The map $t \mapsto u(\alpha)$ gives an injection of function fields $k(t) \to k(\alpha)$. Now if $Q = (x_Q(t),y_Q(t)) \in E(k(t)) \setminus \{O\}$, then
	$$Q' = (x_Q(u(\alpha)),y_Q(u(\alpha))) \in E'(k(\alpha)),$$
	since $E'$ was obtained from $E$ by the same substitution. It's now clear from the injection of function fields that if $Q_1' = Q_2'$ then $Q_1 = Q_2$. Indeed if $f(t) = x_{Q_1}(t) - x_{Q_2}(t)$ evaluates to 0 under the map $t \mapsto u(\alpha)$ then $f(t)$ is identically zero; we can argue similarly for the $y$-coordinates.

	\item From the formulas above and claim 2, we have a commutative diagram
	
	\begin{center}\begin{tikzcd}
E(k(t)) \arrow[r] \arrow[d, "\sigma_{t_0}"] & E'(k(\alpha)) \arrow[d, "\sigma_{\alpha_0}"] \\
E_{t_0}(k) \arrow[r, "\sim"] & E'_{\alpha_0}(k)
\end{tikzcd}\end{center}
where the top arrow is an injection and the bottom arrow is the identity map. Hence if $\sigma_{t_0}(Q) = O_{t_0}$, then $Q$ maps to $O'_{\alpha_0}$ going both ways on the diagram. But going right then down is an injection, so $Q = O$.

\end{enumerate}
\end{proof}

\begin{remark}
You can change variables on $E'$ (preserving Weierstrass form) and still preserve the statements above. For the commutative diagram in the proof of statement 3, instead of the bottom arrow being equality it becomes an isomorphism, and the top arrow is also changed by an isomorphism.
\end{remark}

\begin{remark}
One is only able to use Proposition \ref{prop:torsion_division} to check injectivity of specialization maps for $t_0 \in k$ with the property that $t_0 = u(\alpha_0)$ for some $\alpha_0 \in k$. In particular, these values of $t_0$ often form a Hilbert subset of $k$, and thus Proposition \ref{prop:torsion_division} cannot always be applied to all $t_0 \in k$.
\end{remark}

The benefit of Proposition \ref{prop:torsion_division} is that one always has a 2-torsion point in $E'(\mathbb Q(\alpha))$. Thus one can use Theorem \ref{thm:GT_2-isogeny} (or, in the unlikely case that the polynomial defining $C_2$ is a cyclic cubic over $k(t)$, Theorem \ref{thm:GT_2-descent}) on $E'$ to make statements about the injectivity of specialization maps for $E$ despite the fact that $E$ has no nontrivial $\mathbb Q(t)$-rational 2-torsion points. In particular, in contrast with Algorithm $\ref{algorithm}$, generators of $E(\mathbb Q(t))$ do not need to be known to do this. We illustrate this with the example 
$$E: y^2 = x^3-t^2x+t^2.$$
Set
$$C_2 : a^3-t^2a+t^2 = 0.$$
Using Sage, we obtain the isomorphism of function fields
\begin{align*}
\mathbb Q(C) &\cong \mathbb Q(\alpha) \\
t &\mapsto \frac{1}{\alpha-\alpha^3} \\
a &\mapsto \frac{1}{1-\alpha^2}.
\end{align*}
Via this isomorphism, we obtain the new elliptic curve
$$E': y^2 = x^3-\frac{1}{(\alpha-\alpha^3)^2}x+\frac{1}{(\alpha-\alpha^3)^2}.$$
Setting $x = (\alpha-\alpha^3)^{-2}X$ and $y = (\alpha-\alpha^3)^{-3}Y$, we obtain
$$E'': Y^2 = X^3-(\alpha-\alpha^3)^2X+(\alpha-\alpha^3)^4$$
with $2$-torsion point $(\alpha^4-\alpha^2,0)$. This elliptic curve has discriminant
$$-16\alpha^6(\alpha-1)^6(\alpha+1)^6(3\alpha^2-4)(3\alpha^2-1)^2.$$
Applying Theorem \ref{thm:GT_2-isogeny} to $E''$ now yields the following statement.

\begin{proposition} \label{prop:GT_example_notorsion}
Let $E$ be as above and $t_0$ be a rational number of the form \break $t_0 = 1/(\alpha_0-\alpha_0^3)$ for some rational number $\alpha_0$. Let $\Phi$ be the set of irreducible factors of
$$-\alpha^6(\alpha-1)^6(\alpha+1)^6(3\alpha^2-4)(3\alpha^2-1)^2$$
in $\mathbb Z[\alpha]$ along with $-1$ - more explicitly, set
$$\Phi = \{-1, \alpha, \alpha-1, \alpha+1, 3\alpha^2-4, 3\alpha^2-1\}.$$ Suppose that, for each product $h(\alpha)$ of some nonempty subset of the elements of $\Phi$, the rational number $h(\alpha_0)$ is not a square. Then the specialization map $\sigma_{t_0}$ is injective.
\end{proposition}
\begin{proof}
Theorem \ref{thm:GT_2-isogeny}, Remark \ref{rem:discriminant}, Proposition \ref{prop:torsion_division}.
\end{proof}
For example, this can be used to show that specialization at $t_0 = 8/15$ (corresponding to $\alpha_0 = -3/2$) is injective. Indeed, the relevant factors here are
$$3/2 = -\alpha_0, 5/2 = -(\alpha_0-1), 1/2 = -(\alpha_0+1), 23/4 = 3\alpha_0^2-1, 11/4 = 3\alpha_0^2-4$$
and no product of these is a square. On the other hand, we cannot use this to decide whether or not specialization at $t_0 = 1/6$ (corresponding to $\alpha_0 = -2$) is injective because one of the factors is $1 = -(\alpha_0+1)$.

Next, we generalize Proposition \ref{prop:torsion_division} to give a method of introducing a $k(\alpha)$-rational 2-division point of any $P \in E(k(t)) \setminus E[2](k(t))$ which has no $k(t)$-rational 2-division points. This is a generalization in the sense that the 2-torsion points are precisely the 2-division points of $O$. We will need the ``2-division polynomial of $P$" for this; that is, the polynomial $d_{2,P}$ with the property that $d_{2,P}$ has a root in $k(t)$ if and only if $P$ is divisible by 2 in $E(k(t))$. For details on where this polynomial comes from, see \cite{Billingsley-algorithm}.

\begin{proposition} \label{prop:rational_division} Let $k$ be a number field. Let $E/k(t)$ be an elliptic curve given by the Weierstrass equation
$$y^2 = x^3+A(t)x+B(t),$$
and fix $P = (x_P(t), y_P(t)) \in E(k(t))\setminus E[2](k(t))$ such that $P$ is not divisible by 2 in $E(K)$. Let $\phi(t,x)$ be an irreducible factor of $d_{2,P}(t,x)$ such that
$$C_P: \phi(t,a) =0$$
is rational over $k$. Fix an isomorphism of function fields
\begin{align*}
k(C_P) &\cong k(\alpha) \\
t &\mapsto u(\alpha) \\
a &\mapsto v(\alpha).
\end{align*}
Then the elliptic curve 
$$E': y^2 = x^3+A(u(\alpha))x+B(u(\alpha))$$
defined over $k(\alpha)$ has the following properties.
\begin{enumerate}
	\item $P' = (x_P(u(\alpha)),y_P(u(\alpha))) \in E'(k(\alpha))$ is divisible by 2 in $E'(k(\alpha))$.
	\item The function field isomorphism gives an embedding $E(k(t)) \subset E'(k(\alpha)).$
	\item Let $\alpha_0 \in k$ and set $t_0 = u(\alpha_0)$. Let $M < E(k(t))$ be a subgroup. If the specialization map $\sigma'_{\alpha_0}$ for $E'$ is injective on the image of $M$ via the embedding above, then the specialization map $\sigma_{t_0}|_M$ for $E$ is injective.
\end{enumerate}
\end{proposition}
\begin{proof}
	\begin{enumerate}
	\item Since $\phi(t,x) \in k(t)$ vanishes at $x = a$, we have that $\phi(u(\alpha),x)$ vanishes at $x = v(\alpha)$. Notice that $y_P(u(\alpha)) \neq 0$ since $y_P(t) \neq 0.$ Hence $P'$ is not 2-torsion, and its 2-division polynomial has the root $x = v(\alpha)$. Hence $P'$ is divisible by 2 in $k(\alpha)$.
	\end{enumerate}
Mutatis mutandis, the proof for statements 2 and 3 follows from that of Proposition \ref{prop:torsion_division}.
\end{proof}

\begin{remark} If the 2-torsion curve is rational, one will be able to find many injective specialization maps using Proposition \ref{prop:torsion_division}. If any of these maps happens to be surjective or even have a finite cokernel, one can use generators of the specialization to find generators of $E(k(t))$, thus making it easy to check injectivity for any other $t_0$. The utility of Proposition \ref{prop:rational_division} lies in that the 2-torsion curve may not be rational, but some other curve $C_P$ might be.

\end{remark}

To illustrate how this can be used, let's return to the example 
$$E: y^2 = x^3-t^2x+t^2.$$
We reproduce the result of the calculation in Example \ref{ex:linear_algebra}; that is, that the specialization map for $t_0 = 27$ is injective, but this time without using generators of $E_{27}(\mathbb Q)$. In this example, we cannot use any methods from earlier sections of this article due to a lack of 2-torsion. Additionally, because $P = (t,t)$ specializes to $(27,27) = [2](-9,81)$ (and thus $\closure[1]{\sigma_{27}}$ is not injective), we can use neither Algorithm \ref{algorithm} nor the method just discussed above which combines Theorem \ref{thm:GT_2-isogeny} with Proposition \ref{prop:torsion_division}. We use Proposition \ref{prop:rational_division} as follows. We first find an elliptic curve $E'$ and a point $R' \in E(\mathbb Q(\alpha))$ such that $2R' = P'$. Then, by replacing $P'$ by $R'$, we examine a subgroup $M$ of $E'(\mathbb Q(\alpha))$ which contains a copy of $E(\mathbb Q(t))$. We then hope that the inclusion $M \to E'(\mathbb Q(\alpha))$ satisfies the conditions of Proposition \ref{prop:grouptheory}. If this is true, we may attempt to show that specialization is injective on $M$ using Algorithm \ref{algorithm}, thereby showing that $\sigma_{27}$ is injective. To begin, we have the curve
$$C_P : f_P(t,a) = a^4+2t^2a^2-8t^2a+t^4-t(4a^3-4t^2a+4t^2) = 0$$
defined by the 2-division polynomial of $P$. View $E$ as an elliptic curve over the function field $\mathbb Q(C_P)$ of $C_P$ (which contains $\mathbb Q(t)$). Now, after a bit of searching with Sage, $E(\mathbb Q(C_P))$ has the point 
$$R = \left(a,\frac{a^3-3a^2t+at^2+t^3-2t^2}{2t}\right)$$
such that $[2]R = P$. Sage gives the isomorphism of function fields
\begin{align*}
\mathbb Q(C_P) &\cong \mathbb Q(\alpha) \\
t &\mapsto \frac{4\alpha^3(\alpha + 2)}{(\alpha^2+2\alpha-1)^2} = u(\alpha)\\
a &\mapsto \frac{4\alpha^2(\alpha + 2)}{(\alpha^2+2\alpha-1)^2} = v(\alpha) 
\end{align*}
from which we obtain the new elliptic curve
$$E' : y^2 = x^3 - \left( \frac{4\alpha^3(\alpha + 2)}{(\alpha^2+2\alpha-1)^2}\right)^2x + \left( \frac{4\alpha^3(\alpha + 2)}{(\alpha^2+2\alpha-1)^2}\right)^2.$$

Set $Q = (0,t)$ (and recall that $E(\mathbb Q(t))$ is generated by $P$ and $Q$; see the discussion at the start of \S 4.3.1.). Through this change of variables, we have
\begin{align*}
P' &= \left( \frac{4\alpha^3(\alpha + 2)}{(\alpha^2+2\alpha-1)^2}, \frac{4\alpha^3(\alpha + 2)}{(\alpha^2+2\alpha-1)^2} \right), \\
Q' &= \left( 0,  \frac{4\alpha^3(\alpha + 2)}{(\alpha^2+2\alpha-1)^2} \right), \\
R' &= \left( \frac{4\alpha^2(\alpha + 2)}{(\alpha^2+2\alpha-1)^2}, \frac{4\alpha^3(\alpha+2)(\alpha^2-3)}{(\alpha^2+2\alpha-1)^3} \right).
\end{align*}
In particular, in $E'(\mathbb Q(\alpha))$ we have $2R' = P'$, and $E'$ is an elliptic curve defined over a rational function field over $\mathbb Q$.

We can now use Algorithm \ref{algorithm} on the subgroup $M$ of $E'(\mathbb Q(\alpha))$ generated by $R'$ and $Q'$. We omit the details (which are easily verified using Sage) of showing that the inclusion $M \to E'(\mathbb Q(\alpha))$ satisfies the conditions of Proposition \ref{prop:grouptheory}. Since $2R' = P'$, $M$ contains a copy of $E(\mathbb Q(t))$ by statement 3 of Proposition \ref{prop:rational_division}. We set $\alpha_0 = -3$, because $u(-3) = 27.$ Now, using Sage, we see that the points $R'_{\alpha_0}, Q'_{\alpha_0},$ and $(R'+Q')_{\alpha_0}$ are not divisible by 2 in $E'_{\alpha_0}(\mathbb Q)$. Additionally, the curve $E'_{\alpha_0}$ has no $\mathbb Q$-rational 2-torsion. Hence we conclude that the specialization map $\sigma'_{\alpha_0}$ is injective, so that the specialization map $\sigma_{27}$ is injective.

\begin{remark} The specialization at $t_0 = 7$ can be shown to be injective using the same method as Example \ref{ex:linear_algebra}. Despite this, the method we just outlined (using specialization at $\alpha_0 = -1$) still fails to show that the map is injective. Indeed, in $E'_{-1}(\mathbb Q)$ we have 
$$[2](1,1) = (R'+Q')_{-1},$$
so that $\closure[1]{\sigma'_{-1}}$ fails to be injective on $M$.
\end{remark}

\subsection*{Acknowledgements}

I would like to thank Edray Goins and Donu Arapura for their insightful discussions regarding the contents of this paper. I also would like to thank Kenji Matsuki for contributing significant revisions and clarifications for an early draft, and the referee for numerous corrections and suggestions that streamline the presentation of the results.


\normalsize

\bibliography{refs}
\bibliographystyle{amsalpha}





\normalsize
\baselineskip=17pt








\end{document}